\documentclass[12pt, a4paper]{amsart}
\usepackage[utf8]{inputenc}
\usepackage[T1]{fontenc}
\usepackage[english]{babel}
\usepackage{amsmath,amsfonts,amssymb,latexsym,amscd}
\usepackage{cite}
\theoremstyle{plain}
\newtheorem{theorem}{Theorem}
\newtheorem{lemma}{Lemma}

\newtheorem{corollary}{Corollary}
\theoremstyle{definition}
\newtheorem{definition}{Definition}

\usepackage[all]{xy}

\usepackage{hyperref}

\begin{document}

\title{Equivariant Fibrations} 

\author{Pavel S. Gevorgyan}

\address{Moscow State Pedagogical University}

\email{pgev@yandex.ru}

\subjclass {55R05, 55R65, 55R91}

\keywords{equivariant covering homotopy, equivariant fibration, equivariant homotopy, $H$-fixed point, $H$-orbit space, normal $G$-covering, weakly locally trivial $G$-fibration.}

\begin{abstract}
In this paper, we study equivariant Hurewicz fibrations, obtain their internal characteristics, and prove theorems on relationship between equivariant fibrations and fibrations generated by them. Local and global properties of equivariant fibrations are examined. An equivariant analog of the Hurewicz theorem on passing from local fibrations to global fibrations is proved. A classification of equivariant fibrations with the property of uniqueness of a covering path is given.
\end{abstract}

\maketitle


\section{Introduction}
Let \( G \) be a topological group, \( E \), \( B \), and \( X \) be arbitrary \( G \)-spaces, and \( p : E \to B \) and \( f : X \to B \) be some equivariant mappings. The problem of lifting an equivariant mapping \( f \) to \( E \) consists of the search for an equivariant mapping \( \tilde{f} : X \to E \) such that \( f = p \circ \tilde{f} \). For this problem to become a well-posed problem of \( G \)-homotopy category, the mapping \( p : E \to B \) must satisfy an additional condition.

We say that an equivariant mapping \( p : E \to B \) possesses the property of \textit{equivariant covering homotopy} with respect to a \( G \)-space \( X \) if for arbitrary equivariant mappings \( \tilde{f} : X \to E \) and \( F : X \times I \to B \) such that \( F \circ i_0 = p \circ \tilde{f} \), where \( i_0 : X \to X \times I \) is the embedding defined by the formula \( i_0(x) = (x, 0) \), \( x \in X \), there exists an equivariant mapping \( \tilde{F} : X \times I \to E \) such that the following diagram is commutative:
$$
\xymatrix
{
X \ar[rr]^{\tilde{f}} \ar[d]_{i_0} & & E \ar[d]^{p} \\
X\times I \ar[rr]_{F} \ar@{-->}[rru]^{\tilde{F}} & & B 
}
$$
i.e.  $p\circ \tilde{F} =F$ and $\tilde{F}\circ i_0=\tilde{f}$. 

If \( p : E \to B \) possesses the property of equivariant covering homotopy with respect to a \( G \)-space \( X \) and \( f, g : X \to B \) be \( G \)-homotopic equivariant mappings, then \( f \) can be lifted to \( E \) if and only if \( g \) can be lifted to \( E \). In other words, the lift of an equivariant mapping \( f : X \to B \) is a property of the \( G \)-homotopic class of this mapping.

The property of equivariant covering homotopy leads to the concept of equivariant fibration. An equivariant mapping \( p : E \to B \) is called an \textit{equivariant Hurewicz fibration} or a \textit{Hurewicz \( G \)-fibration} if \( p \) possesses the property of equivariant covering homotopy with respect to any \( G \)-space \( X \).

In this paper, we examine equivariant Hurewicz fibrations. Using equivariant covering functions and extended equivariant covering functions, we obtain inner characteristics of these fibrations (see Theorems~\ref{th-1} and~\ref{th-1-1}). We establish relationships of \( G \)-fibrations with generated fibrations (Theorems~\ref{th-5-0}, \ref{th-5}, and \ref{th-5-1} and Corollaries~\ref{cor-1} and~\ref{cor-2}) and examine their local and global properties. We prove an equivariant analog of Hurewicz's theorem on the equivalence of local and global \( G \)-fibrations in the case of paracompact base (Theorem~\ref{th-4}). We introduce the concept of locally trivial \( G \)-fibration (Definition~\ref{def-1}) and show that the class of such fibrations is wider than the class of locally trivial \( G \)-fibration, but in narrower than the class of Hurewicz \( G \)-fibrations in the case of paracompact base (Theorem~\ref{th-7}). We prove that under some additional condition for the base of a fibration, the classes of weakly locally trivial \( G \)-fibrations and Hurewicz \( G \)-fibrations coincide (Theorem~\ref{th-3}). Finally, we propose a classification of \( G \)-fibrations with the property of uniqueness of a covering paths (Theorem~\ref{th-6}). 

Concepts, definitions, and results from equivariant topology and fibration theory used in this paper can be found in \cite{james-segal,May-Sig,Spener,tom Dieck_1}.


\section{\texorpdfstring{\( G \)-Fibrations}{G-Fibrations} and their inner characteristics}

Let \( G \) be a compact topological group, \( E \) and \( B \) be \( G \)-spaces, and \( p : E \to B \) be an equivariant mapping.

The action of the group \( G \) on \( B \) induces the natural action of the group \( G \) on the path space \( B^I \) by the formula \((g\alpha)(t) = g\alpha(t)\) for arbitrary \( g \in G \), \(\alpha \in B^I\), and \( t \in I \). Consider the subspace \(\Delta \subset E \times B^I\) defined by the formula
\[
\Delta = \{(e,\alpha) \in E \times B^I \mid \alpha(0) = p(e)\}.
\]
Obviously, \(\Delta\) is an invariant subset of the \( G \)-space \( E \times B^I \).

\begin{definition}
An equivariant mapping \(\lambda : \Delta \to E^I\) satisfying the conditions
\[
\lambda(e,\alpha)(0) = e, \quad [p \circ \lambda(e,\alpha)](t) = \alpha(t),
\]
is called an \textit{equivariant covering function} or a \textit{covering \( G \)-function} for \( p \).
\end{definition}

The following theorem states the relationship between equivariant Hurewicz fibrations and equivariant covering functions. It provides an inner characteristic of an equivariant Hurewicz fibration.

\begin{theorem}\label{th-1}
An equivariant mapping \( p : E \to B \) is a Hurewicz \( G \)-fibration if and only if \( p \) possesses an equivariant covering function.
\end{theorem}

\textbf{Proof.} Let \(\lambda : \Delta \to E^I\) be an equivariant covering function for \(p\). Consider an arbitrary \(G\)-space \(X\), an arbitrary equivariant mapping \(\tilde{f} : X \times 0 \to E\), and an arbitrary equivariant homotopy \(F : X \times I \to B\), \(F(x, 0) = p(\tilde{f}(x, 0))\). Note that for arbitrary \(x \in X\), the formula \(F_x(t) = F(x, t)\) defines a path \(F_x \in B^I\). It is easy to see that the homotopy \(\tilde{F} : X \times I \to E\) defined by the formula \(\tilde{F}(x, t) = \lambda(\tilde{f}(x, 0), F_x)(t)\) is the required equivariant covering homotopy.

Conversely, assume that \(p : E \to B\) is a Hurewicz \(G\)-fibration. Let \(X = \Delta\). We consider the mappings
\[
\tilde{f} : \Delta \times 0 \to E, \quad \tilde{f}[(e, \alpha), 0] = e, \quad F : \Delta \times I \to B, \quad F[(e, \alpha), t] = \alpha(t).
\]
One can easily verify that \(\tilde{f}\) and \(F\) are equivariant mappings and
\[
F[(e, \alpha), 0] = \alpha(0) = p(e) = (p \circ \tilde{f})[(e, \alpha), 0].
\]
Therefore, there exists an equivariant covering homotopy \(\tilde{F} : \Delta \times I \to E\) for the homotopy \(F\):
\[
\tilde{F}[(e, \alpha), 0] = \tilde{f}, \quad p \circ \tilde{F} = F.
\]
Now we introduce the mapping
\[
\lambda : \Delta \to E^I, \quad \lambda(e, \alpha)(t) = \tilde{F}[(e, \alpha), t].
\]
It is easy to see that \(\lambda\) is an equivariant covering function for \(p\).

Let \(W \subset B^I\) be an invariant subset of the \(G\)-space \(B^I\). Consider the set \(\Delta_W \subset E \times B^I \times I\) defined by the formula
\[
\Delta_W = \{(e, \alpha, s) \in E \times W \times I \mid \alpha(s) = p(e)\}.
\]
Obviously, \(\Delta_W\) is an invariant subset of the \(G\)-space \(E \times B^I \times I\).

\begin{definition}\label{def-ext}
An equivariant mapping \(\Lambda : \Delta_W \to E^I\) satisfying the conditions
\[
\Lambda(e, \alpha, s)(s) = e, \quad [p \circ \Lambda(e, \alpha, s)](t) = \alpha(t),
\]
is called an \textit{extended equivariant covering function} or \textit{extended covering \( G \)-function} over \(W\).
\end{definition}

The following lemma shows that the existence of an extended covering \( G \)-function over \( B^I \) is equivalent to the existence of an equivariant covering function.

\begin{lemma}
An equivariant mapping \( p : E \to B \) possesses an equivariant covering function if and only if there exists an extended covering \( G \)-function over \( B^I \).
\end{lemma}

The proof of this lemma is similar to that in the nonequivariant case (see \cite[Lemma 9, p. 124]{Spener}).

Theorem 1 and Lemma 1 immediately imply the following assertions.

\begin{theorem}\label{th-1-1}
An equivariant mapping \( p : E \to B \) is a Hurewicz \( G \)-fibration if and only if there exists an extended covering \( G \)-function over \( B^I \).
\end{theorem}

Theorems \ref{th-1} and \ref{th-1-1} are of great practical importance; in what follows, they will be applied to prove that the given equivariant mapping is a Hurewicz \( G \)-fibration.


\section{\texorpdfstring{\( G \)-Fibrations}{G-Fibrations} and fibrations generated by them}

Let \( H \) be a closed subgroup of a compact group \( G \). Then any \( G \)-space is also an \( H \)-space and any \( G \)-mapping is an \( H \)-mapping. Thus, there exists a natural covariant functor from the category \( G \)-TOP into the category \( H \)-TOP. The following natural question arises: Does this functor preserve the property of an equivariant mapping to be an equivariant Hurewicz fibration? The answer is given by the following theorem.

\begin{theorem}\label{th-5-0}
Let a \( G \)-mapping \( p : E \to B \) be a Hurewicz \( G \)-fibration. Then for an arbitrary closed subgroup \( H \) of the compact group \( G \), the mapping \( p : E \to B \) is a Hurewicz \( H \)-fibration.
\end{theorem}

\begin{proof}
By Theorem \ref{th-1}, there exists a covering \( G \)-function \( \lambda : \Delta \to E^I \) for the \( G \)-mapping \( p \). Since \( \lambda : \Delta \to E^I \) is also an \( H \)-mapping, \( \lambda \) is a covering \( H \)-function for the \( H \)-mapping \( p \). Therefore, \( p : E \to B \) is a Hurewicz \( H \)-fibration by Theorem \ref{th-1}.
\end{proof}

In particular, Theorem \ref{th-5-0} implies the following assertion.

\begin{corollary}\label{cor-1}
An arbitrary Hurewicz \( G \)-fibration \( p : E \to B \) is a Hurewicz fibration.
\end{corollary}

Let \( X \) and \( Y \) be \( G \)-spaces and \( H \) be a closed subgroup of the group \( G \). The set
\[
X^H = \{ x \in X \mid hx = x \text{ for all } h \in H \}
\]
is called the \textit{space of \( H \)-fixed points} of the \( G \)-space \( X \). An equivariant mapping \( f : X \to Y \) generates a natural mapping \( p^H \) between the spaces of \( H \)-fixed points of \( X^H \) and \( Y^H \).

\begin{theorem}\label{th-5}
Let \( p : E \to B \) be a Hurewicz \( G \)-fibration. Then for an arbitrary closed subgroup \( H \) of the compact group \( G \), the induced mapping \( p^H : E^H \to B^H \) between the spaces of \( H \)-fixed points is a Hurewicz fibration.
\end{theorem}

\begin{proof}
By Theorem \ref{th-1}, for the \( G \)-mapping \( p \), there exists a covering \( G \)-function \( \lambda : \Delta \to E^I \). Then \( \lambda^H : \Delta^H \to (E^H)^I \) is a covering function for the mapping \( p^H : E^H \to B^H \). Therefore, \( p^H : E^H \to B^H \) is a Hurewicz fibration. \qedhere
\end{proof}

In the case where \( G \) is a compact Lie group and \( E \) and \( B \) be \( G \)-CW-complexes, the converse assertion is also valid (see \cite[Theorem 9]{Gev-Jim}).

The property of an equivariant mapping to be a Hurewicz \( G \)-fibration is preserved under passing to spaces of \( H \)-orbits.

\begin{theorem}\label{th-5-1}
Let \( p : E \to B \) be a Hurewicz \( G \)-fibration. Then for an arbitrary closed subgroup \( H \) of the compact group \( G \), the induced \( G \)-mapping \( p^* : E|H \to B|H \) between the spaces of \( H \)-orbits is a Hurewicz \( G \)-fibration.
\end{theorem}

\begin{proof}
Let \( \lambda : \Delta \to E^I \) be a covering \( G \)-function for the \( G \)-mapping \( p \) (see Theorem \ref{th-1}). Consider the set
\[
\Delta_H = \{(e^*, \alpha^*) \in E|H \times (B|H)^I \mid \alpha^*(0) = p^*(e^*) \}
\]
and introduce the mapping
\[
\lambda^* : \Delta_H \rightarrow (E|H)^I, \quad \lambda^*(e^*, \alpha^*)(t) = (\lambda(e, \alpha)(t))^*.
\]
The mapping \(\lambda^*\) is well defined; it is a covering \(G\)-function for \(p^* : E|H \rightarrow B|H\). Therefore, by Theorem \ref{th-1}, \(p^* : E|H \rightarrow B|H\) is a Hurewicz \(G\)-fibration.
\end{proof}

In the particular case where \(H = G\), Theorem \ref{th-5-1} implies the following assertion.

\begin{corollary}\label{cor-2}
Let \(p : E \rightarrow B\) be a Hurewicz \(G\)-fibration. Then the induced mapping \(p^* : E|G \rightarrow B|G\) between the space of orbits is a Hurewicz fibration.
\end{corollary}


\section{Local and global \texorpdfstring{\( G \)-fibrations}{G-fibrations}}

An equivariant mapping \(p : E \rightarrow B\) is called a \textit{local Hurewicz \(G\)-fibration} if for an arbitrary point \(b \in B\), there exists an invariant neighborhood \(U\) such that 
\[
p|_{p^{-1}(U)} : p^{-1}(U) \rightarrow U
\]
is a Hurewicz \(G\)-fibration.

An open cover \(\mathcal{U} = \{U_i \mid i \in \mathcal{I}\}\) of a \(G\)-space \(X\) consisting of invariant subsets \(U_i\) is called a \textit{\(G\)-cover}.

A \(G\)-cover \(\mathcal{U} = \{U_i \mid i \in \mathcal{I}\}\) is called a \textit{normal \(G\)-cover} if for any \(U_i\), there exists a function \(\chi_i : X \rightarrow I\) such that \(\chi_i(gx) = \chi_i(x)\) for arbitrary \(x \in X\) and \(g \in G\) and \(U_i = \{x \in X \mid \chi_i(x) \neq 0\}\). The function \(\chi_i\) is called the \textit{characteristic \(G\)-function} of the open invariant subset \(U_i \subset X\).

\begin{lemma}\label{lemm-2}
Let \(G\) be a compact group. Then an arbitrary paracompact \(G\)-space \(X\) possesses a locally finite normal \(G\)-cover.
\end{lemma}

\begin{proof}
Let \(\mathcal{U} = \{U_i \mid i \in \mathcal{I}\}\) be a locally finite normal cover consisting of open invariant subsets of a paracompact \(G\)-space \(X\) and \(h_i : X \rightarrow I\) be the characteristic function (not equivariant) of the subset \(U_i, i \in \mathcal{I}\). Then the function \(\chi_i : X \rightarrow I\) defined by the formula \(\chi_i(x) = \int h_i(gx)dg\), where \(\int\) is the Haar integral, is the characteristic \(G\)-function of the invariant subset \(U_i\). Therefore, \(\mathcal{U} = \{U_i \mid i \in \mathcal{I}\}\) is a locally finite normal \(G\)-cover of the paracompact \(G\)-space \(X\). 
\end{proof}

The following theorem shows that under sufficiently general assumptions, a local Hurewicz \(G\)-fibration is a global Hurewicz \(G\)-fibration.

\begin{theorem}\label{th-06}
Let \(G\) be a compact group, \(p : E \rightarrow B\) be an equivariant mapping, and \(\mathcal{U} = \{U_i \mid i \in \mathcal{I}\}\) be a locally finite normal \(G\)-cover of the \(G\)-space \(B\). If \(p|_{p^{-1}(U_i)} : p^{-1}(U_i) \rightarrow U_i\) is a Hurewicz \(G\)-fibration for arbitrary \(U_i \in \mathcal{U}\), then \(p\) is a Hurewicz \(G\)-fibration.
\end{theorem}

\begin{proof}
For any finite set of indices \(i_1, \ldots, i_k \in \mathcal{I}\) (\(k \geq 1\)), we introduce the set \(V_{i_1 \ldots i_k} \subset B^I\) as follows:
\[
V_{i_1 \ldots i_k} = \left\{ \alpha \in B^I \mid \alpha \left( \left[ \frac{j-1}{k}, \frac{j}{k} \right] \right) \subset U_{i_j},\ j = 1, \ldots, k \right\}.
\]
Clearly, \(V_{i_1 \ldots i_k}\) is an invariant subset of the \(G\)-space \(B^I\). We prove that there exists an extended covering \(G\)-function over each \(V_{i_1 \ldots i_k}\). By Theorem \ref{th-5}, there exists an extended covering \(G\)-function \(\Lambda_j : \Delta_{U_{i_j}} \rightarrow E^I\) over each \(U_{i_j},\ j = 1, \ldots, k\). Consider an arbitrary element \((e, \alpha, s) \in \Delta_{V_{i_1 \ldots i_k}}\). Assume that \(s \in [(n - 1)/k, n/k]\) for some natural \(n\). Let \(\alpha_j \in B^I, j = 1, \ldots, k\), be the path defined by the formula
\[
\alpha_j(t) =
\begin{cases}
\alpha \left( \frac{j-1}{k} \right), & t \in \left[ 0, \frac{j-1}{k} \right], \\
\alpha(t), & t \in \left[ \frac{j-1}{k}, \frac{j}{k} \right], \\
\alpha \left( \frac{j}{k} \right), & t \in \left[ \frac{j}{k}, 1 \right].
\end{cases}
\]
Now we define an extended covering \(G\)-function \(\Lambda : \Delta_{V_{i_1 \ldots i_k}} \rightarrow E^I\) as follows. First, we define \(\Lambda(e, \alpha, s)\) on the segment \([(n - 1)/k, n/k]\) by the formula
\[
\Lambda(e, \alpha, s)(t) = \Lambda_n(e, \alpha_n, s), \quad t \in \left[ \frac{n-1}{k}, \frac{n}{k} \right].
\]
Next, we define \(\Lambda(e,\alpha,s)\) consecutively on the segments \([(n-2)/k,(n-1)/k]\), \([(n-3)/k,(n-2)/k]\), \ldots, \([0,1/k]\) by the formula
\[
\Lambda(e,\alpha,s)(t) = \Lambda_j \left( \Lambda(e,\alpha,s) \left( \frac{j}{k} \right), \alpha_j, \frac{j}{k} \right), \quad t \in \left[ \frac{j-1}{k}, \frac{j}{k} \right], \quad j = n-1, n-2, \ldots, 1.
\]
Finally, we define \(\Lambda(e,\alpha,s)\) consecutively on the segments \([n/k,(n+1)/k]\), \([(n+1)/k,(n+2)/k]\), \ldots, \([(k-1)/k,1]\) by the formula
\[
\Lambda(e,\alpha,s)(t) = \Lambda_{j+1} \left( \Lambda(e,\alpha,s) \left( \frac{j}{k} \right), \alpha_{j+1}, \frac{j}{k} \right), \quad t \in \left[ \frac{j}{k}, \frac{j+1}{k} \right], \quad j = n, n+1, \ldots, k-1.
\]
Obviously, the equivariant mapping \(\Lambda : \Delta_{V_{i_1\ldots i_k}} \to E^I\) constructed satisfies Definition \ref{def-ext}, i.e., is an extended covering \(G\)-function over \(V_{i_1\ldots i_k}\).

The family \(\{V_{i_1\ldots i_k};\ k=1,2,\ldots\}\) is a normal \(G\)-cover of the \(G\)-space \(B^I\) and contains an inscribed \(G\)-cover \(\{W_\mu\}\) with linearly ordered set of indices \(\mathcal{M}\), which is also locally finite (see \cite[Lemma 3.4]{Dugundji}, or \cite{Hurewicz}).

Let \(\chi_\mu : B^I \to I\) be the characteristic \(G\)-function of an invariant subset \(W_\mu \subset B^I\), i.e., \(\chi_\mu(g\alpha) = \chi_\mu(\alpha)\) for arbitrary \(g \in G\) and \(\alpha \in B^I\) and let \(W_\mu = \{\alpha \in B^I \mid \chi_\mu(\alpha) \neq 0\}\). We denote the extended covering \(G\)-function over the invariant set \(W_\mu \subset B^I\) by \(\Lambda_\mu\). Consider an arbitrary element \((e,\alpha) \in \Delta \subset E \times B^I\). Let \(\mu_1 < \mu_2 < \ldots < \mu_n\) be the set of all indices such that \(\alpha \in W_{\mu_i},\ i=1,\ldots,n\). We define the points \(t_i,\ i=1,\ldots,n\), of the unit segment \(I\) by the formula
\[
t_i = \frac{\sum\limits_{j=1}^i \chi_{\mu_j}(\alpha)}{\sum\limits_{j=1}^n \chi_{\mu_j}(\alpha)}, \quad i=1,\ldots,n.
\]
Consider the path \(\lambda(e,\alpha) \in E^I\) defined by the following formulas consecutively on the segments \([0,t_1]\), \([t_1,t_2], \ldots, [t_{n-1},1]\):
\[
\lambda(e,\alpha)(t) = \Lambda_{\mu_1}(e,\alpha,0)(t), \quad t \in [0,t_1],
\]
\[
\lambda(e,\alpha)(t) = \Lambda_{\mu_{i+1}}(\lambda(e,\alpha)(t_i), \alpha, t_i)(t), \quad t \in [t_i, t_{i+1}], \quad i=1,2,\ldots,n-1.
\]
Obviously, \(\lambda(e,\alpha)(0) = e\) and \([p \circ \lambda(e,\alpha)](t) = \alpha(t)\). Moreover, \(\lambda : \Delta \to E^I\) is an equivariant mapping since \(\Lambda_\mu, \mu \in \mathcal{M}\), are equivariant. Therefore, \(\lambda : \Delta \to E^I\) is a covering \(G\)-function for \(p\), i.e., \(p : E \to B\) is a Hurewicz \(G\)-fibration by Theorem \ref{th-1}.
\end{proof}

Theorem \ref{th-06} and Lemma \ref{lemm-2} immediately imply the following assertion.

\begin{theorem}\label{th-4}
Let \(G\) be a compact group and \(B\) be a paracompact \(G\)-space. An equivariant mapping \(p : E \to B\) is a Hurewicz \(G\)-fibration if and only if it is a local Hurewicz \(G\)-fibration.
\end{theorem}

Thus, in the case of a paracompact base \(B\), the notions of local and global \(G\)-fibrations coincide. Theorem \ref{th-4} is an equivariant analog of the well-known Hurewicz theorem (see \cite{Hurewicz}).


\section{Weakly locally trivial \texorpdfstring{\( G \)-fibrations}{G-fibrations}}

Let \(G\) be a compact group, \(B\) and \(F\) be arbitrary \(G\)-spaces, and \(p : B \times F \to B\) be the projection onto the first factor: \(p(x,y) = x\). The equivariant mapping \(p : B \times F \to B\) is a Hurewicz \(G\)-fibration. Indeed, let \(X\) be an arbitrary \(G\)-space, \(\tilde{f}: X \to B \times F\) be an equivariant mapping, and \(F : X \times I \to B\) be an equivariant homotopy such that \(F(x,0) = p\tilde{f}(x),\ x \in X\). Obviously, the mapping \(\tilde{F} : X \times I \to B \times F\) defined by the formula \(\tilde{F}(x,t) = (F(x,t), \mathrm{pr}_2(\tilde{f}(x,0)))\) is an equivariant covering homotopy for \(F : X \times I \to B\). The projection \(p : B \times F \to B\) is called the \textit{trivial \(G\)-fibration}.

An equivariant mapping \(p : E \to B\) is called a \textit{locally trivial \(G\)-fibration} if for an arbitrary point \(b \in B\), there exists an invariant neighborhood \(U\) such that \(p|_{p^{-1}(U)} : p^{-1}(U) \to U\) is the trivial \(G\)-fibration.

Theorem \ref{th-4} implies that if \( B \) is a paracompact \( G \)-space, then the locally trivial \( G \)-fibration \( p : E \to B \) is a Hurewicz \( G \)-fibration.

\begin{definition}\label{def-1}
An equivariant mapping \( p : E \to B \) is called a \textit{weakly locally trivial \( G \)-fibration} if for an arbitrary point of the \( G \)-space \( B \), there exist an open invariant neighborhood \( U \) of this point and an equivariant mapping \( \omega : U \times p^{-1}(U) \to p^{-1}(U) \) such that  

(i) \( p \circ \omega(b, e) = b \) for all \((b, e) \in U \times p^{-1}(U)\),  

(ii) \( \omega(p(e), e) = e \) for all \( e \in p^{-1}(U) \).
\end{definition}

Condition (i) means that the composition  
\[
U \times p^{-1}(U) \xrightarrow{\omega} p^{-1}(U) \xrightarrow{p} U
\]  
is the projection onto the first factor. Obviously, a locally trivial \( G \)-fibration \( p : E \to B \) is weakly locally trivial \( G \)-fibration. In contrast to the trivial \( G \)-fibration, the equivariant mapping \( \omega : U \times p^{-1}(U) \to p^{-1}(U) \) in Definition \ref{def-1} is not an equivariant homeomorphism in general.

\begin{lemma}\label{lemm-1}
A weakly locally trivial \( G \)-fibration \( p : E \to B \) is a local Hurewicz \( G \)-fibration.
\end{lemma}

\begin{proof}
Let \( b_0 \in B \) be an arbitrary point. Consider an invariant neighborhood \( U \) of this point and an equivariant mapping \( \omega : U \times p^{-1}(U) \to E \) satisfying the conditions (i) and (ii) of Definition \ref{def-1}. We prove that  
\[
p|_{p^{-1}(U)} : p^{-1}(U) \to U
\]  
is a Hurewicz \( G \)-fibration. We define the mapping  
\[
\lambda : \Delta \to (p^{-1}(U))^I, \quad \lambda(e, \alpha)(t) = \omega(\alpha(t), e).
\]  
It is easy to verify that \(\lambda\) is a covering \( G \)-function for the \( G \)-mapping \( p|_{p^{-1}(U)} \). 
\end{proof}

Lemma \ref{lemm-1} and Theorem \ref{th-4} immediately imply the following assertion.

\begin{theorem}\label{th-7}
Let \( p : E \to B \) be a weakly locally trivial \( G \)-fibration, where \( B \) is a paracompact \( G \)-space. Then \( p : E \to B \) is a Hurewicz \( G \)-fibration.
\end{theorem}

The converse assertion is valid under sufficiently weak restrictions imposed on the \( G \)-space \( B \) (see Theorem~\ref{th-3}).

\begin{definition}\label{def-EULC}
A \( G \)-space \( X \) is said to be \textit{equivariantly uniformly locally retractable} if for an arbitrary point \( x_0 \in X \), there exist an invariant neighborhood \( U \) of this point and an equivariant mapping \( \sigma : U \times U \times I \to X \) such that the following conditions are fulfilled:  

(i) \( \sigma(x, y, 0) = x \) and \( \sigma(x, y, 1) = y \) for all \((x, y) \in U \times U\),  

(ii) \( \sigma(x, x, t) = x \) for all \( x \in U \) and \( t \in I \).
\end{definition}

The following lemma shows that there exist many equivariantly uniformly locally retractable \( G \)-spaces.

\begin{lemma}\label{lem-1}
An arbitrary \( G \)-ANR-space is equivariantly uniformly locally retractable.
\end{lemma}

The proof of this lemma is simple and is left to the reader.

In particular, Lemma \ref{lem-1} implies that \( G \)-CW-complexes are equivariantly uniformly locally retractable \( G \)-spaces.

\begin{theorem}\label{th-2}
Let \( B \) be an equivariantly uniformly locally retractable \( G \)-space. Then \( p : E \to B \) is a local Hurewicz \( G \)-fibration if and only if \( p : E \to B \) is a weakly locally trivial \( G \)-fibration.
\end{theorem}

\begin{proof}
The sufficiency without any additional restrictions for the \( G \)-space \( B \) was proved in Lemma \ref{lemm-1}.

Prove the necessity. Let \( p : E \to B \) be a local Hurewicz \( G \)-fibration. For an arbitrary point \( b_0 \in B \), consider an equivariant neighborhood \( V \) of the point \( b_0 \) such that  
\[
p|_{p^{-1}(V)} : p^{-1}(V) \to V
\]  
is a Hurewicz \( G \)-fibration. Let \(\lambda : \Delta \to (p^{-1}(V))^I\) be a covering \( G \)-function of the mapping \( p|_{p^{-1}(V)} \). Since \( B \) is an equivariantly uniformly locally retractable \( G \)-space, there exist an invariant neighborhood \( U \) of the point \( b_0, U \subset V \), and an equivariant mapping \(\sigma : U \times U \times I \to B\) such that the conditions \(\sigma(b, b', 0) = b, \sigma(b, b', 1) = b'\) and \(\sigma(b, b, t) = b\) are fulfilled for all \(b, b' \in U\) and \(t \in I\). Moreover, the invariant neighborhood \(U\) can be chosen in such a way that \(\sigma(U \times U \times I) \subset V\): it suffices to note that \((b_0, b_0, I) \in \sigma^{-1}(V)\) and apply the compactness of \(I\).

Now let \((b, e) \in U \times p^{-1}(U)\). Consider the path \(\tilde{\alpha} \in B^I\) defined by the formula \(\tilde{\alpha}(t) = \sigma(p(e), b, t)\). Note that \(\tilde{\alpha}\) joins the point \(p(e)\) with the point \(b\). Introduce the equivariant mapping
\[
\omega : U \times p^{-1}(U) \to p^{-1}(U), \quad \omega(b, e) = \lambda(e, \tilde{\alpha})(1).
\]
It is easy to verify that \(\omega\) satisfies the conditions (i) and (ii) of Definition \ref{def-1}. 
\end{proof}

Theorems \ref{th-4} and \ref{th-2} immediately imply the following important theorem.

\begin{theorem}\label{th-3}
Let \(B\) be a paracompact, equivariantly uniformly locally retractable \(G\)-space. Then \(p : E \to B\) is a Hurewicz \(G\)-fibration if and only if \(p : E \to B\) is a weakly locally trivial \(G\)-fibration.
\end{theorem}


\section{\texorpdfstring{\( G \)-Fibrations}{G-Fibrations} with the property of uniqueness of a covering path}

We say that a mapping \(p : E \to B\) possesses the property of uniqueness of a covering path if for any two paths \(\alpha, \alpha' : I \to E\) satisfying the conditions \(\alpha(0) = \alpha'(0)\) and \(p \circ \alpha = p \circ \alpha'\), the equality \(\alpha = \alpha'\) holds.

Consider arbitrary \(G\)-spaces \(E\) and \(B\), where \(G\) is a compact topological group.

\begin{lemma}\label{lemma-2}
Let an equivariant mapping \(p : E \to B\) be a Hurewicz fibration with the property of uniqueness of a covering path. Then \(p : E \to B\) is a Hurewicz \(G\)-fibration.
\end{lemma}

\begin{proof}
Let \(X\) be an arbitrary \(G\)-space and \(\tilde{f} : X \to E\) and \(F : X \times I \to B\) be arbitrary equivariant mappings satisfying the condition \(F(x, 0) = (p \circ \tilde{f})(x)\). Since \(p : E \to B\) is a Hurewicz fibration, there exists a homotopy \(\tilde{F} : X \times I \to E\) such that \(\tilde{F}(x, 0) = \tilde{f}(x)\) and \(p \circ \tilde{F} = F\). To complete the proof, it suffices to verify that the covering homotopy \(\tilde{F}\) is in fact an equivariant covering homotopy. Consider the paths \(\alpha, \alpha' : I \to E\) defined by the formulas
\[
\alpha(t) = \tilde{F}(gx, t), \quad \alpha'(t) = g\tilde{F}(x, t),
\]
where \(g \in G\), \(x \in X\) are arbitrary fixed elements. Note that
\[
\alpha(0) = \tilde{F}(gx, 0) = \tilde{f}(gx) = g\tilde{f}(x), \quad \alpha'(0) = g\tilde{F}(x, 0) = g\tilde{f}(x).
\]
Thus, \(\alpha(0) = \alpha'(0)\). Moreover,
\[
(p \circ \alpha)(t) = (p \circ \tilde{F})(gx, t) = F(gx, t) = gF(x, t),
\]
\[
(p \circ \alpha')(t) = p(g\tilde{F}(x, t)) = gp(\tilde{F}(x, t)) = g(p \circ \tilde{F})(x, t) = gF(x, t).
\]
Thus, \(p \circ \alpha = p \circ \alpha'\). Therefore, \(\alpha = \alpha'\), i.e., \(\tilde{F}(gx, t) = g\tilde{F}(x, t)\), since \(p : E \to B\) possesses the property of uniqueness of a covering path.
\end{proof}

Lemma \ref{lemma-2} and Corollary \ref{cor-1} immediately imply the following theorem.

\begin{theorem}\label{th-6}
An equivariant mapping \(p : E \to B\) with the property of uniqueness of a covering path is a Hurewicz \(G\)-fibration if and only if it is a Hurewicz fibration.
\end{theorem}

In other words, in the class of all equivariant mappings with the property of uniqueness of a covering path, the notion of Hurewicz \(G\)-fibration coincides with the notion of a Hurewicz fibration.

\end{document}